\documentclass[12pt]{article}
\usepackage[usenames]{color}
\usepackage{amsmath}
\usepackage{amsfonts}
\usepackage{verbatim}
\usepackage{graphicx}
\usepackage{graphics}
\usepackage{epsfig}
\usepackage{amssymb}
\usepackage{amscd}
\usepackage{amsthm}

\newtheorem{theorem}{Theorem}[subsection]

\newtheorem{corollary}[theorem]{Corollary}

\newtheorem{proposition}[theorem]{Proposition}
\newtheorem{lemma}[theorem]{Lemma}
\newtheorem{definition}[theorem]{Definition}

\newtheorem{example}[theorem]{Example}
\newtheorem{remark}[theorem]{Remark}
\numberwithin{equation}{section}

\title{The ladder crystal}

\author{Chris Berg\\
 University of California at Davis, Davis, California, USA \\ Fields Institute, Toronto, ON, Canada\\
\small \texttt{cberg@fields.utoronto.edu}\\
}

\begin{document}


\newlength\cellsize \setlength\cellsize{18\unitlength}
\savebox2{%
\begin{picture}(18,18)
\put(0,0){\line(1,0){18}}
\put(0,0){\line(0,1){18}}
\put(18,0){\line(0,1){18}}
\put(0,18){\line(1,0){18}}
\end{picture}}
\newcommand\cellify[1]{\def\thearg{#1}\def\nothing{}%
\ifx\thearg\nothing
\vrule width0pt height\cellsize depth0pt\else
\hbox to 0pt{\usebox2\hss}\fi%
\vbox to 18\unitlength{
\vss
\hbox to 18\unitlength{\hss$#1$\hss}
\vss}}
\newcommand\tableau[1]{\vtop{\let\\=\cr
\setlength\baselineskip{-16000pt}
\setlength\lineskiplimit{16000pt}
\setlength\lineskip{0pt}
\halign{&\cellify{##}\cr#1\crcr}}}
\savebox3{%
\begin{picture}(15,15)
\put(0,0){\line(1,0){15}}
\put(0,0){\line(0,1){15}}
\put(15,0){\line(0,1){15}}
\put(0,15){\line(1,0){15}}
\end{picture}}
\newcommand\expath[1]{%
\hbox to 0pt{\usebox3\hss}%
\vbox to 15\unitlength{
\vss
\hbox to 15\unitlength{\hss$#1$\hss}
\vss}}

\maketitle

\begin{abstract}
In this paper I introduce a new description of the crystal $B(\Lambda_0)$ of $\widehat{\mathfrak{sl}_\ell}$. 
As in the Misra-Miwa model of $B(\Lambda_0)$, the nodes of this crystal are indexed by partitions and 
the $i$-arrows correspond to adding a box of residue $i$. I then show that the two models are equivalent by interpreting
the operation of regularization introduced by James as a crystal isomorphism.
\end{abstract}

\section{Introduction}

The main goal of this paper is to give a combinatorial description of the crystal of the basic representation of
$\widehat{\mathfrak{sl}_\ell}$. Misra and Miwa previously gave such a description which involved $\ell$-regular 
partitions, and which I will denote as $reg_\ell$. My description, denoted $ladd_\ell$, satisfies the 
following properties:

\begin{itemize}
\item The nodes of $ladd_\ell$ are partitions, and there is an $i$-arrow from $\lambda$ to $\mu$ only when the difference $\mu \setminus \lambda$ is a box of residue $i$. 
\item $reg_\ell \cong ladd_\ell$ and this crystal  isomorphism yields an interesting 
bijection on the nodes.  The map being used for the isomorphism has been well studied \cite{BOX}, but 
never before in the context of a crystal isomorphism.
\item The partitions which are nodes of $ladd_\ell$ can be identified by a simple combinatorial condition.
\end{itemize}

\subsection{Background and Previous Results}
Let $\lambda$ be a partition of $n$ (written $\lambda \vdash n$) and $\ell \geq 3$ be an integer. We will 
use the convention $(x,y)$ to denote the box which sits in the $x^{\textrm{th}}$ row and the $y^{\textrm{th}}$
 column of the Young diagram of $\lambda$. $\mathcal{P}$ will denote the set 
of all partitions. An \textit{$\ell$-regular partition} is one in which no part occurs $\ell$ or more times.
To each box $(x,y)$ in a Young diagram of $\lambda$, the \textit{residue} of that box is the difference $y-x$ taken modulo $\ell$.

For two partitions $\lambda$ and $\mu$ of $n$, 
we say that $\lambda \leq \mu$ if $\sum_{j=1}^i \lambda_j \leq \sum_{j=1}^i \mu_j$ for all $i$. This order is
 usually called the \textit{dominance order}.

The \textit{hook length} of the $(a,c)$ box of $\lambda$ is defined to be the number of boxes 
to the right of or below the box $(a,c)$, including the box $(a,c)$ itself. It will be denoted 
\textit{$h_{(a,c)}^{\lambda}$}.
The \textit{arm} of the $(a,c)$ box of $\lambda$ is defined to be the number of boxes to the
 right of the box $(a,c)$, \textit{not} including the box $(a,c)$. It will be denoted $arm(a,c)$.  

\subsubsection{Ladders}
For any box $(a,b)$ in the Young diagram of 
$\lambda$, the \textit{ladder} of $(a,b)$ is the set of all positions $(c,d)$ which satisfy $\frac{c-a}{d-b} =
 \ell-1$ and $c,d >0$.

\begin{remark}
The definition implies that two positions in the same ladder will share the same residue. An $i$-ladder will be a
 ladder which has residue $i$.
\end{remark}
 
\begin{example}
Let $\lambda = (3,3,1)$, $\ell = 3$. Then there is a 1-ladder which contains the positions 
$(1,2)$ and $(3,1)$, and a different 1-ladder which has the position $(2,3)$ in $\lambda$ and the 
positions $(4,2)$ and $(6,1)$ not in $\lambda$. In the picture below, lines are drawn through the 
different 1-ladders.

\begin{center}
$ \begin{array}{cc}
\put (13,-38){\line(1,2){27}}
\put (68,17){\line(-1,-2){52}}

&
\tableau{0&1&2\\
2&0&1\\
1}
\end{array}$
\end{center}
\end{example}

\subsubsection{Regularization} Regularization is a map which takes a partition to a $p$-regular partition.  For a given $\lambda$, move all of the boxes up to the top of their respective ladders. 
The result is a partition, and that partition is called the \textit{regularization} of $\lambda$, and is 
denoted $\mathcal{R} \lambda$. The following theorem contains facts about regularization originally 
due to James \cite{Jold} (see also \cite{JM}).
\begin{theorem}\label{reg_prop} Let $\lambda$ be a partition. Then
\begin{itemize}
\item $\mathcal{R} \lambda$ is $\ell$-regular;
\item $\mathcal{R} \lambda = \lambda$ if and only if $\lambda$ is $\ell$-regular.
\end{itemize}
\end{theorem}

Regularization provides us with an equivalence relation on the set of partitions. Specifically, we say $\lambda \sim \mu$ if $\mathcal{R} \lambda = \mathcal{R} \mu$. The equivalence classes are called \textit{regularization classes}, and the class of a partition $\lambda$ is denoted $\mathcal{RC}(\lambda) := \{ \mu \in \mathcal{P} : \mathcal{R}\mu = \mathcal{R}\lambda \}$.

\begin{example}\label{reg_example} Let $\lambda = (2,2,2,1,1,1)$ and let $\ell = 3$. Then $\mathcal{R}\lambda = (3,3,2,1)$. Also, 
\[\mathcal{RC}(\lambda) = \{ (2, 2, 2, 1, 1, 1),
(2, 2, 2, 2, 1),
(3, 2, 1, 1, 1, 1),\]
\[
(3, 2, 2, 2),
(3, 3, 1, 1, 1),
(3, 3, 2, 1)    \}\].

\begin{center}
$\begin{array}{lcr}
\tableau{0&1\\
2&0\\
1&2\\
0\\
2\\
1}
&
\displaystyle
\xrightarrow{\mathcal{R}}
&
\tableau{0&1&2\\
2&0&1\\
1&2\\
0}

\end{array}$
\end{center}
\end{example}

\subsection{Summary of results from this paper}
 In Section \ref{old_crystal} we recall the description 
of the crystal $B(\Lambda_0)$  of $\widehat{\mathfrak{sl}_{\ell}}$ involving $\ell$-regular partitions. 
In Section \ref{new_crystal} we give our new description of the crystal $B(\Lambda_0)$. Section \ref{locked} gives a new procedure for finding an 
inverse for the map of regularization. Section \ref{crystal_lemmas} reinterprets the classical
 crystal rules in the combinatorial framework of the new crystal rules. Section \ref{reg_and_crystal} 
contains the proof that the two descriptions of the crystal $B(\Lambda_0)$ are isomorphic, an isomorphism 
being given by regularization.

\section{Classical Description of $reg_\ell$}\label{old_crystal}
\subsection{Introduction}
In this section, we recall a description of the crystal graph $B(\Lambda_0)$ 
first described by Misra and Miwa \cite{MM}. 

\subsection{Crystals}
We start by giving a notion of a crystal. 
Informally, we will say that a crystal of $\widehat{\mathfrak{sl}_\ell}$ is a set $B$ together with operators $\widetilde{e}_i,\widetilde{f}_i: B \rightarrow B \cup \{0\}$ for each $i \in \{0,1,\dots,\ell-1\}$ satisfying:

\begin{itemize}
\item $\widetilde{e}_i a = b$ if and only if $\widetilde{f}_i b = a$ for $a,b \in B$.
\item For each $b \in B$ and $i \in  \{0,1,\dots,\ell-1\}$ there exists an $n$ such that $\widetilde{f}_i^n b = \widetilde{e}_i^n b= 0$.
\end{itemize}

We view the crystal $B$ as a graph with nodes coming from $B$ and an $i$ colored directed edge from $a$ to $b$ whenever $\widehat{f}_i a= b$.

\begin{remark}

The crystal we study, $B(\Lambda_0)$ can be interpreted as the \textit{crystal basis} of the basic representation $V(\Lambda_0)$ of $\widehat{\mathfrak{sl}_\ell}$. It is not the intention of the author to give a full description of the theory of crystal basis; such definitions can be found in Hong and Kang's book \cite{HK} or in the work of Kashiwara \cite{K}. Instead, I will focus a well known combinatorial description of $B(\Lambda_0)$ and give a combinatorial isomorphism to my own combinatorial construction.
\end{remark}

 \subsection{Classical description of the crystal $reg_\ell$} We look at the crystal $B(\Lambda_0)$ of the irreducible highest weight module $V(\Lambda_0)$ of the affine Lie algebra $\widehat{\mathfrak{sl}_{\ell}}$ (also called the basic representation of $\widehat{\mathfrak{sl}_{\ell}}$). In the Misra-Miwa description,
the nodes of $reg_\ell$ are $\ell$-regular partitions. The set of nodes will be denoted $B := \{ \lambda \in \mathcal{P} : \, \lambda \textrm{ is } \ell \textrm{-regular} \} $. We will describe the arrows of $reg_\ell$ below. 

We view the Young diagram for $\lambda$ as a set of boxes, with the residue $b-a \mod \ell$ written into the box $(a,b)$. A position in $\lambda$ is said to be a removable $i$-box if it has residue $i$ and after removing that box the remaining diagram is still a partition. A position not in $\lambda$ is an addable $i$-box if it has residue $i$ and adding that box to $\lambda$ yields a partition.

For a fixed $i$, ($0 \leq i < \ell$), we place  $-$ in each removable $i$-box and $+$ in each addable $i$-box. The $i$-signature of $\lambda$ is the word of $+$ and $-$'s in the diagram for $\lambda$, written from bottom left to top right. The reduced $i$-signature is the word obtained after repeatedly removing from the $i$-signature all adjacent pairs $- +$. The resulting word will now be of the form $+ \dots +++--- \dots -$. The positions corresponding to $-$'s in the reduced $i$-signature are called \textit{normal $i$-boxes}, and the positions corresponding to $+$'s are called \textit{conormal $i$-boxes}. $\varepsilon_i(\lambda)$ is defined to be the number of normal $i$-boxes of $\lambda$, and $\varphi_i(\lambda)$ is defined to be the number of conormal $i$-boxes. If there are any $-$ signs in the reduced $i$-signature, the position corresponding to the leftmost one is called the \textit{good $i$-box} of $\lambda$. If there are any $+$ signs in the reduced $i$-signature, the position corresponding to the rightmost one is called the \textit{cogood $i$-box}. All of these definitions can be found in Kleshchev's book \cite{Kl}.

We recall the action of the crystal operators on $B.$ The crystal operator $\widetilde{e}_{i}: B \xrightarrow{i} B \cup \{0\}$ assigns to a partition  $\lambda$ the partition $\widetilde{e}_{i}(\lambda) = \lambda \setminus x$, where $x$ is the good $i$-box of $\lambda$. If no such box exists, then $\widetilde{e}_{i}(\lambda)=0$. It can be easily shown that $\varepsilon_i(\lambda) = max\{k : \widetilde{e}_{i}^k \lambda \neq 0\}$.

Similarly, $\widetilde{f}_{i}: B \xrightarrow{i} B \cup \{0\}$ is the operator which assigns to a partition  $\lambda$ the partition $\widetilde{f}_{i}(\lambda) = \lambda \cup x$, where $x$ is the cogood $i$-box of $\lambda$. If no such box exists, then $\widetilde{f}_{i}(\lambda)=0$. It can be easily shown that $\varphi_i(\lambda) = max\{k : \widetilde{f}_{i}^k \lambda \neq 0\}$.

For $i \in \mathbb{Z}/ \ell \mathbb{Z}$, we write $\lambda \xrightarrow{i} \mu$ to stand for $\widetilde{f}_{i} \lambda = \mu$. We say that there is an $i$-arrow from $\lambda$ to $\mu$. Note that $\lambda \xrightarrow{i} \mu$ if and only if $\widetilde{e}_{i} \mu = \lambda$. A maximal chain of consecutive $i$-arrows will be called an $i$-string. We note that the empty partition $\emptyset$ is the unique highest weight node of the crystal ( i.e. it is the unique $\ell$-regular partition satisfying $\widetilde{e}_{i} \emptyset = 0 $ for every $i \in \mathbb{Z} / \ell \mathbb{Z}$.) For a picture of a part of this crystal graph, see \cite{LLT} for the cases $\ell = 2$ and $3$. 
  
For the rest of this paper, $\varphi = \varphi_i(\lambda)$ and $\varepsilon = \varepsilon_i(\lambda)$.

\begin{center}
\begin{figure}
\centering
\includegraphics{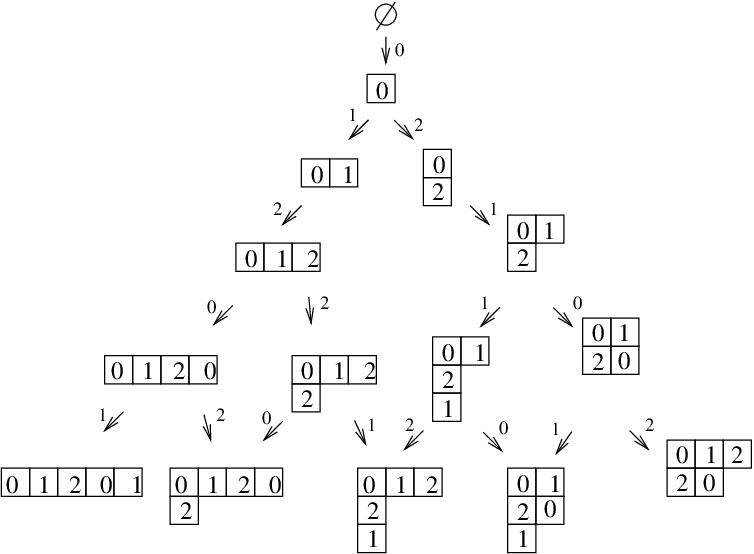}
\caption{The first 6 levels of $reg_\ell$ for $\ell =3$}
\end{figure}
\end{center}

\section{The Ladder Crystal: $ladd_\ell$}\label{new_crystal}
\subsection{The ladder crystal}
For $0\leq i < \ell$, we define operators $\widehat{f}_{i}$ (and $\widehat{e}_i$) acting on partitions, taking a partition of $n$ to a partition of $n+1$ (resp. $n-1$) (or 0) in the following manner. Given $\lambda \vdash n$, first draw all of the $i$-ladders of $\lambda$ onto its Young diagram. Label any addable $i$-box with a $+$, and any removable $i$-box with a $-$. Now, write down the word of $+$'s and $-$'s by reading from leftmost $i$-ladder to rightmost $i$-ladder and reading from top to bottom on each ladder. This is called the \textit{ladder $i$-signature} of $\lambda$. From here, cancel any adjacent $-+$ pairs in the word, until you obtain a word of the form $+\dots+-\dots-$. This is called the \textit{reduced ladder $i$-signature of $\lambda$}. All positions associated to a $-$ in the reduced ladder $i$-signature are called \textit{ladder normal $i$-boxes} and all positions associated to a $+$ in the reduced ladder $i$-signature are called \textit{ladder conormal $i$-boxes}. The position associated to the leftmost $-$ is called the \textit{ladder good $i$-box} and the position associated to the rightmost $+$ is called the \textit{ladder cogood $i$-box}. Then we define $\widehat{f}_{i} \lambda$ to be the partition obtained by adding the ladder cogood $i$-box to $\lambda$. If no such box exists, then $\widehat{f}_{i} \lambda = 0$. Similarly, $\widehat{e_i} \lambda$ is the partition $\lambda$ with the ladder good $i$-box removed. If no such box exists, then $\widehat{e_i} \lambda = 0$. We then define $\widehat{\varphi}_i(\lambda)$ to be the number of ladder conormal $i$-boxes of $\lambda$ and $\widehat{\varepsilon}_i(\lambda)$ to be the number of ladder normal $i$-boxes. It can be shown that $\widehat{\varphi_i}(\lambda) = \max \{ k : \widehat{f}_i^k \lambda \neq 0 \}$ and that $\widehat{\varepsilon_i}(\lambda) = \max \{ k : \widehat{e}_i^k \lambda \neq 0 \}$. For the rest of the paper, $\widehat{\varphi} = \widehat{\varphi}_i(\lambda)$ and $\widehat{\varepsilon} = \widehat{\varepsilon}_i(\lambda)$. 

\begin{remark}
The only difference between $\widehat{f}_i$ and $\tilde{f}_i$ is in how the boxes are ordered. $\tilde{f}_i$ reads boxes from bottom to top whereas $\widehat{f}_i$ reads boxes down ladders, starting with the leftmost ladder. 
\end{remark}

We now define $ladd_\ell$ to be the connected component obtained by starting with the empty partition and using the crystal operators $\widehat{f}_i$. 

\begin{remark}
It remains to be shown that this directed graph is a crystal. To see this, we will show that it is isomorphic to $reg_\ell$. 
\end{remark}

\begin{remark} 
Those who study crystals know that they also have a weight function assigned to the nodes of the graph. Here, just as in the classical description $reg_\ell$, the weight of a partition $\lambda$ is given by $wt(\lambda) = \Lambda_0 - \sum_i a_i \alpha_i$, where $a_i$ denotes the number of boxes of $\lambda$ of residue $i$.
\end{remark}

\begin{example}\label{5311111}
Let $\lambda = (5,3,1,1,1,1,1)$ and $\ell = 3$. Then there are four addable 2-boxes for $\lambda$. In the leftmost 2-ladder (containing position (2,1)) there are no addable (or removable) 2-boxes. In the next 2-ladder (containing position (1,3)) there is an addable 2-box in position (3,2). In the next 2-ladder (containing position (2,4)), there are two addable 2-boxes, in positions (2,4) and (8,1). In the last drawn 2-ladder (containing position (1,6)) there is one addable 2-box, in position (1,6). There are no removable 2-boxes in $\lambda$. Therefore the ladder 2-signature (and hence reduced ladder 2-signature) of $\lambda$ is $+_{(3,2)}+_{(2,4)}+_{(8,1)}+_{(1,6)}$ (Here, we have included subscripts on the $+$ signs so that the reader can see the correct order of the $+$'s). Hence ${\widehat{f}_2} \lambda = (6,3,1,1,1,1,1)$, $({\widehat{f}_2})^2 \lambda = (6,3,1,1,1,1,1,1)$, $({\widehat{f}_2})^3 \lambda = (6,4,1,1,1,1,1,1)$ and $({\widehat{f}_2})^4 \lambda = (6,4,2,1,1,1,1,1)$. $({\widehat{f}_2})^5 \lambda = 0$.

\begin{center}
$\begin{array}{cc}
\put (10,-29){\line(1,2){23.5}}
\put (10,-82){\line(1,2){50}}
\put (15,-126){\line(1,2){72}}
\put (41,-126){\line(1,2){72}}
\put (106,6){2}
\put (70,-12){2}
\put (34,-30){2}
\put (16, -120){2}

& \tableau{0&1&2&0&1\\
2&0&1\\
1\\
0\\
2\\
1\\
0
}
\end{array}$
\end{center}
\end{example}

\begin{center}
\begin{figure}
\centering
\includegraphics{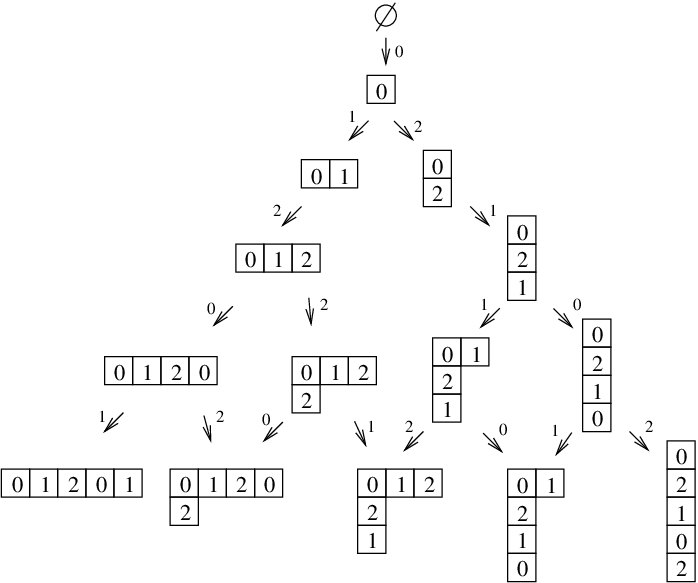}
\caption{The first 6 levels of $ladd_\ell$ for $\ell =3$}
\end{figure}
\end{center}

From this description, it is not obvious that this is a crystal. However, we will soon show that it is isomorphic to $reg_\ell$. 

We end this section by proving a simple property of $\widehat{e}_i$ and $\widehat{f}_i$.
\begin{lemma}\label{e_inverse_f} $\widehat{e}_i \lambda = \mu$ if and only if $\widehat{f}_i \mu = \lambda$.
\end{lemma}
\begin{proof}

Suppose $\widehat{f}_i \mu = \lambda$. It is enough to show that the ladder $i$-good box of $\lambda$ is the ladder $i$-cogood box of $\mu$. This is true because adding the $i$-cogood box of $\mu$ does not cause cancellation in the reduced ladder $i$-signature (if it did, then there would have been a + in the reduced ladder $i$-signature to the right of the cogood position).
Thus $\widehat{e}_i \lambda = \mu$. The other direction is similar.

\end{proof}

\section{Deregularization}\label{locked}

The goal of this section is to provide a method for finding the smallest partition in dominance order in a given regularization class. It is nontrivial to show that a smallest partition exists. We use this result to show that our new description of the crystal $B(\Lambda_0)$ has nodes which are least dominant in their regularization classes. All of the work of this section is inspired by Brant Jones of UC Davis, who gave the first definition of a locked box.

\subsection{Locked Boxes}

Finding all of the partitions which belong to a regularization class is not easy. The definition of locked boxes below formalizes the concept that some boxes in a partition cannot be moved down their ladders if one requires that the new diagram remain a partition.

\begin{definition}
For a partition $\lambda$, we label boxes of $\lambda$ as \textit{locked} by the following procedure:
\begin{enumerate}
\item If a box $x$ has a locked box directly above it (or is on the first row) and every unoccupied position in $\mathfrak{L}x$, lying below $x$, has an unoccupied position directly above it then $x$ is locked. Boxes locked for this reason are called type I locked boxes.
\item If a box $y$ is locked, then every box to the left of $y$ in the same row is also locked. Boxes locked for this reason are called type II locked boxes.

\end{enumerate}
Boxes which are not locked are called \textit{unlocked}.  \end{definition}

\begin{remark}
Locked boxes can be both type I and type II.
\end{remark}

\begin{example}
Let $\ell = 3$ and let $\lambda = (7,5,4,3,1,1)$. 
Then labelling the locked boxes for $\lambda$ with an $L$ and the unlocked boxes with a $U$ yields the picture below.

\begin{center}
$\tableau{L&L&L&L&L&L&L\\
L&L&L&L&U\\
L&L&U&U\\
L&L&U\\
L\\
L}
$
\end{center}
\end{example}

The following lemmas follow from the definition of locked boxes.
\begin{lemma}\label{locked_up}
If $(a,b)$ is locked and $a>1$ then $(a-1,b)$ is locked. Equivalently, all boxes which sit below an unlocked box in the same column are unlocked.

\end{lemma}

\begin{proof}
If $(a,b)$ is a type I locked box then by definition $(a-1,b)$ is locked. If $(a,b)$ is a type II locked box and not type I then there exists a $c$ with $c>b$ such that $(a,c)$ is a type I locked box. But then by definition of type I locked box, $(a-1,c)$ is locked. Then $(a-1,b)$ is  a type II locked box.
\end{proof}

\begin{lemma}\label{locked_up_ladder}
If there is a locked box in position $(a, b)$ and there is a box in position $(a - \ell +1, b+1)$ then the box $(a - \ell+1, b+1)$ is locked. 
\end{lemma}

\begin{proof}
To show this, suppose that $(a- (\ell-1), b+1)$ is unlocked. Then let $(c, b+1)$ be the highest unlocked box in column $b+1$. The fact that $(c,b+1)$ is unlocked implies that there is an unoccupied position below it, on the same ladder with a box immediately above it. Then the box $(c+\ell-1, b)$ violates the type I locked condition, and it will not have a locked box directly to the right of it ($(c, b+1)$ is unlocked, so  $(c+\ell-1, b+1)$ will be unlocked if it is occupied, by Lemma \ref{locked_up}). Hence $(c+\ell-1, b)$ is unlocked, so $(a, b)$ must be unlocked since it sits below $(c+\ell-1, b)$ (by Lemma \ref{locked_up}), a contradiction.
\end{proof}

For two partitions $\lambda$ and $\mu$ in the same regularization class, 
there may be many ways to move the boxes in $\lambda$ on their ladders to obtain $\mu$. 
We define an \textit{arrangement of} $\mu$ \textit{from} $\lambda$ to be a bijection which assigns each box in the 
Young diagram of $\lambda$ to a box in the same ladder of the Young diagram of $\mu$. An arrangement will be denoted by a set
 of ordered pairs $(x,y)$ with $x \in \lambda$ and $y \in \mu$, where both $x$ and $y$ are in the same ladder,
 and each $x \in \lambda$ and $y \in \mu$ is used exactly once. Such a pair $(x,y)$ denotes that the box $x$ 
from $\lambda$ is moved into position $y$ in $\mu$. We also introduce an ordering of boxes on each ladder; for two positions $x$ and $y$ on the same ladder, we say that $x \prec y$ if the position $x$ lies below $y$ on the ladder which they share.

\begin{remark} We introduce the notation $\mathcal{B}(p)$ to denote the position paired with $p$ in an 
arrangement $\mathcal{B}$, i.e. $(p, \mathcal{B}(p)) \in \mathcal{B}$. 
\end{remark}

\begin{example}
$\lambda = (3,3,1,1,1)$ and $\mu = (2,2,2,2,1)$ are in the same regularization class when $\ell = 3$ 
(see Example \ref{reg_example}). One possible arrangement of $\mu$ from $\lambda$ would be 
\[\mathcal{B} = \{   ((1,1),(1,1)), ((1,2),(1,2)),((1,3),(5,1)),((2,1),(2,1)),\] \[ ((2,2),(2,2)), ((2,3),(4,2)), ((3,1),(3,1)),((4,1),(4,1)), ((5,1),(3,2))\}.\] This corresponds to moving the labeled boxes from $\lambda$ to $\mu$  in the corresponding picture below. Note that $\mathcal{B}((5,1)) \succ (5,1)$, $\mathcal{B}((1,3)) \prec (1,3)$ and $\mathcal{B}((2,3)) \prec (2,3)$. 
\begin{center}

$\begin{array}{lr}
\tableau{1&2&3\\
4&5&6\\
7\\
8\\
9}
&
\tableau{1&2\\
4&5\\
7&9\\
8&6\\
3}
\end{array}$
\end{center}
\end{example}

\subsection{Finding the smallest partition in a regularization class}
 For any partition $\lambda$, to find the smallest partition (with respect to dominance order) in a 
regularization class we first label each box of $\lambda$ as either locked or unlocked as above. Then 
we create a new diagram $\mathcal{S}\lambda$ which moves all unlocked boxes down their ladders, while 
keeping these unlocked boxes in order (from bottom to top), while locked boxes do not move. It is
 unclear that this procedure will yield the smallest partition in $\mathcal{RC}(\lambda)$, or even that 
$\mathcal{S} \lambda$ is a partition. In this subsection, we resolve these issues. To shorten notation, we will use $\mathfrak{L}x$ to denote the ladder which a box $x$ sits in.

\begin{proposition}\label{moveboxesdown} Let $\lambda$ and $\mu$ be partitions in the same regularization class. Then there exists an arrangement $\mathcal{D}$ of $\mu$ from $\lambda$ such that for any locked box $x$ of $\lambda$,  $\mathcal{D}(x) \succeq x$. 
\end{proposition}

\begin{proof}

To find a contradiction, we suppose that for any arrangement $\mathcal{C}$ of $\mu$ from $\lambda$ there must be a locked box $a$ such that $\mathcal{C}(a) \prec a$.
 Among all of these boxes, we label a box $x_\mathcal{C} $ which is in the highest row and in the furthest right column amongst the boxes in the highest row. 

Among all arrangements of $\mu$ from $\lambda$, let $\mathcal{D}$ be one which has $x_{\mathcal{D}}$ in the lowest row, and amongst all such in the lowest row also has the leftmost column. Let $x= x_{\mathcal{D}}$. 

We will exhibit a box $w$ on $\mathfrak{L}x$, such that $\mathcal{D}(w) \succeq x$ and either $w \prec x$ or $w \succ x$ and is unlocked. If such a box exists, then letting $\mathcal{A} = (\mathcal{D} \setminus \{(x, \mathcal{D}(x)), (w, \mathcal{D}(w))\}) \cup \{(x, \mathcal{D}(w)), (w, \mathcal{D}(x)) \}$ will yield a contradiction, as $x_{\mathcal{A}}$ will be in a position to the left of and/or below $x$, which contradicts our choice of $\mathcal{D}$. 

If there exists a $w$ in $\mathfrak{L}x$, $ w \prec x$, with $\mathcal{D}(w) \succeq x$ in $\mu$ then we are done. So now we assume:

\begin{equation*}\tag{*}\label{assumption} \mathcal{D}(w) \prec  x \textrm{ for every } w \prec x.
\end{equation*}

There are two cases to consider:

\vspace{10pt}
 \noindent\textbf{Case I:} $x$ is a type II locked box and not a type I locked box. 

In this case, there is a locked box directly to the right of $x$ (definition of a type II lock), which we label $y$. 

If $\mathcal{D}(y) = y$ then some box $w$ must satisfy $\mathcal{D}(w) =x$. The assumption implies that $w\succ x$, so $w$ is unlocked (because $x = x_{\mathcal{D}}$ was the highest locked box which moved down according to $\mathcal{D}$) and we are done.

If $\mathcal{D}(y) \neq y$ then $\mathcal{D}(y) \succ y$ (since it is to the right of $x$ and locked).
 If $\tilde{x} \succ x$, is at the end of its row, then $\tilde{x}$ is unlocked (it is not a type II lock because there are no boxes to the right of it, and it can't be a type I lock because $x$ is not a type I lock and Lemma \ref{locked_up} implies that there is an empty position in $\mathfrak{L}x$ directly below a box). If $\mathcal{D}(\tilde{x}) \succeq x$, then we could use $w = \tilde{x}$ and be done. So otherwise we assume all such $\tilde{x}$ satisfy $\mathcal{D}(\tilde{x}) \prec x$. 
 
Similarly, if $\hat{x} \succ x$ which is directly to the left of a box $\hat{y}$ in $\mathfrak{L}y$, and  $\mathcal{D}(\hat{y}) \prec y$, then $\hat{x}$ must be unlocked (it's not a type II lock because $\hat{y}$ is unlocked, being above $x$, and its not a type I lock because $x$ is not). If $\mathcal{D}(\hat{x}) \succeq x$, then we could use $w=\hat{x}$ and be done. So we assume that all such $\hat{x}$ satisfy $\mathcal{D}(\hat{x}) \prec x$. 

Assuming we cannot find any $w$ by these methods, we let
\begin{itemize}
\item $k$ denote the number of boxes $w \succ x$ in $\lambda$,
\item $j$ denote the number of boxes $w\succ y$, in $\lambda$,
\item $k'$ denote the number of boxes $w \succ x$, in $\mu$,
\item $j'$ denote the number of boxes $w \succ y$, in $\mu$,
\item $m$ denote the number of boxes $w \succ y$, in $\lambda$ which satisfy $\mathcal{D}(b) \prec b$ (i.e. the number of boxes which are of the form $\hat{y}$ above).
\end{itemize}

The number of boxes of the form $\tilde{x}$ is then $k-j$. Also, $j' \geq j-m +1$, since the $m$ boxes
 need not move below $x$, but the box in $y$ moves above $x$. 
$k' \leq k - m - (k-j) = j-m $ since the number of boxes in ladder $x$ will go down by at least $m$ for the boxes of the form $\hat{x}$ and $k-j$ for the boxes of the form $\tilde{x}$ (due to (\ref{assumption})). Hence $k' \leq j-m < j-m+1 \leq j'$.   This is a contradiction, as there must be at least as many boxes $w \succ x$ on the ladder of $x$  as there are $v \succ y$ on the ladder of $y$ for $\mu$ to be a partition.

\vspace{10pt}

\noindent
\textbf{Case II:} $x$ is a type I locked box.

We give $x$ coordinates $(a,b)$.

Since $(a,b)$ is a type I lock, the number of boxes $y \preceq (a,b)$ which are in $\lambda$ must be strictly greater than the number of boxes $z \prec (a-1,b)$ which are in $\lambda$.
Since $\mathcal{D}((a,b)) \prec (a,b)$, some box $p \succ (a-1,b)$ satisfies $\mathcal{D}(p) \prec p$
(since no box $y \prec (a,b)$ satisfies $\mathcal{D}(y) \succ (a,b)$, by (\ref{assumption})). 
Since $(a,b)$ is the highest locked box for which $\mathcal{D}((a,b)) \prec (a,b)$, $p$ must be an unlocked box. 

The existence of $p$ implies that there is an empty position $q \succ (a-1,b)$ in $\lambda$. This is implied by Lemma \ref{locked_up_ladder} above, since if every position $r \succ (a-1,b)$ was occupied, then all of those boxes would be locked. 

Let $m$ be the column which contains the lowest such empty position, which has coordinates $(a-1-(m-b)(\ell-1),m)$. Let $(c_1, d_1)$ be the highest empty position for which $(c_1,d_1) \prec (a,b)$ (we know a position must exist since $\mathcal{D}((a,b))\prec (a,b)$ and (\ref{assumption}) implies that all other positions $r \prec (a,b)$ satisfy $\mathcal{D}(r) \prec (a,b)$). Since $(a,b)$ is a type I lock, $(c_1-1, d_1)$ is also an empty position. We know that $(a-1,b)$ is locked since $(a,b)$ is a type I lock. If $(a-1, b)$ is also a type I lock then the position $(c_1-2, d_1)$ must also be empty. Continuing this, if $(a-k, b)$ were a type I lock for every $0 \leq k \leq \ell-2$ then $(c_1- (\ell-1), d_1)$ would have to be empty, but this would contradict the fact that there should be a box in $(c_1-(\ell-1),d_1 +1)$ since it is in $\mathfrak{L}(a,b)$ and $(c_1, d_1)$ was chosen to be the the highest empty position on $\mathfrak{L}(a,b)$ below $(a,b)$. 

So there exists a $k_1$ so that $(a-k_1,b)$ is a type II lock, with $k_1 \leq \ell-2$. Since $(a-k_1, b)$ is a type II lock, there exists a type I locked box in a position $(a-k_1, n_1)$ with $b < n_1 < m$. Let $(c_2, d_2) \prec (a-k_1,n_1)$ denote the highest empty position ($(c_2, d_2)$ exists because the position in column $d_1$ on $\mathfrak{L}(a-k_1, n_1)$ is empty). Since $(a-k_1,n_1)$ is a type I lock, the postition $(c_2-1, d_2)$ is also empty. Continuing as above, if $(a-k_1-k, n_1)$ were a type I lock for $0 \leq k \leq \ell-2$, then $(c_2- (\ell-1), d_2)$ would be empty, which contradicts our choice of $(c_2,d_2)$. Hence there exists a $k_2 \leq \ell-2$ so that $(a-k_1-k_2 , n_1)$ is a type II lock. This implies that there is a type I lock in some position $(a-k_1-k_2, n_2)$ with $n_1 < n_2 < m$. Continuing this, we get a sequences for $k$ and $n$ with each $k_i \leq \ell-2$ and $b < n_1 < n_2 < \dots < n_i < m$. 

I claim that each of the boxes $(a-\sum_i{k_i}, n_i)$ are in a ladder below the ladder of $(a-1,b)$. If this is the case then the sequences for $k$ and $n$ would eventually have to produce a type I locked box in column $m$ or greater, below row $a-1-(m-b)(\ell-1)$.  This contradicts that $\lambda$ is a partition, since there is no box is position $(a-1-(m-b)(\ell-1),m)$.

To show the claim, we just note that each successive type I locked box comes from moving up at most $\ell-2$ and to the right at least one position. These are clearly in a ladder below the ladder of $(a-1, b)$, since ladders move up $\ell-1$ boxes each time they move one box to the right. In particular, because $(a-\sum k_i, n_j)$ stay in a ladder on or below $\mathcal{L}(a-1,b)$, $a-\sum k_i \geq \ell-1$, so that each of the boxes $(a-\sum k_i, n_j)$ are well defined (not above row $1$).

\end{proof}

\begin{corollary}\label{smallestpartition}
Fix $n\in \mathbb{N}$. Suppose for any partition $\mu \vdash n$, $\mathcal{S} \mu$ is also a partition. Then $\mathcal{S} \lambda$ is the smallest partition (in dominance order) in the regularization class of $\lambda$. Futhermore, all boxes of $\mathcal{S} \lambda$ are locked.
\end{corollary}

\begin{proof}
Let $\mu$ be any partition in the regularization class of $\lambda$. Then by Proposition \ref{moveboxesdown} above, we can choose an arrangement $\mathcal{B}$ of $\mu$ from $\lambda$ in which the only boxes which move down are unlocked boxes of $\lambda$. But $\mathcal{S}\lambda$ requires all unlocked boxes of $\lambda$ to be moved down as far as they can, so $\mathcal{S}\lambda \leq \mu$. For the second statement, assume some box of $\mathcal{S} \lambda$ is unlocked. Then we can apply $\mathcal{S}$ again, to obtain $\mathcal{S}^2\lambda$ which is smaller than $\mathcal{S} \lambda$, which contradicts the first 
statement of this corollary. 
\end{proof}

\begin{lemma}\label{lockedremove}
Let $\lambda$ be a partition whose entire first row is locked, and let $\mu$ be the partition $(\lambda_2, \lambda_3, \dots )$. Then a box $(a,b)$ with $(a>1)$ is locked in $\lambda$ if and only if $(a-1,b)$ is locked in $\mu$. Similarly, suppose that $\lambda$ has unlocked boxes in its last column. Let $\mu$ be a partition obtained by removing any number of the boxes in that column (in such a way that $\mu$ is actually a partition). Then boxes $(c,d) \in \mu$ are locked if and only if they are locked in $\lambda$.
\end{lemma}

\begin{proof}
For the first statement, it is enough to notice that the first row of the partition $\lambda$ just plays the same role as the top of the partition plays in the first locked box condition. For the second statement, it is enough to note that for any box $z$ in the last column of $\lambda$, there are no boxes to the right of $z$ or above $z$, in the same ladder as $z$, which are the two interests of the locked box conditions.
\end{proof}

\begin{proposition}\label{algo_yields_partition}
Let $\lambda$ be a partition. Then $\mathcal{S}\lambda$ is a partition.
\end{proposition}

\begin{proof}

Suppose there was a box $w$ in $\mathcal{S}\lambda$ which was locked in $\lambda$ but has an empty position either directly above it or  directly to the left of it. If this were the case then it would contradict $w$ being locked in $\lambda$, as an unlocked box to the left of $w$ is impossible due to the second locking rule and an unlocked box above $w$ would contradict Lemma \ref{locked_up}.

We let $\mathcal{S}_\lambda$ denote the assignment of boxes in $\lambda$ to boxes in $\mu$ which is followed via the application of $\mathcal{S}$. The proposition will imply that $\mathcal{S}_\lambda$ is in fact an arrangement of $\mathcal{S}\lambda$ from $\lambda$.

 There are two remaining possibilities we must rule out: That an unlocked box in $\lambda$ which was moved down via $\mathcal{S}_\lambda$ has an empty position directly to the left, or that it has an empty position directly above. If we show that the construction never leaves an empty position above a moved unlocked box, then 
as a consequence we can easily prove that we get no empty positions to the left of a moved box.
  
Suppose the construction never leaves an empty position above a moved box. If there was an empty position $x$ directly to the left of a moved box $\mathcal{S}_\lambda(y)$ then below $x$ there would be a box $z$ on $\mathfrak{L}\mathcal{S}_\lambda(y)$. If $z$ was not moved, then it must have been locked, which implies all of the boxes above it were locked, including the box which was in position $x$, contradicting that position $x$ is empty. Otherwise, $z$ moved, so our assumption implies that $z$ has no empty position above it. Applying this procedure again, we can determine that there must be a box directly above the box directly above $z$. Applying this procedure will eventually imply that there must be a box in position $x$. Therefore our goal is to show that moving down all unlocked boxes produces no box below an empty position.
 
  We prove this by induction on $n$, the number of boxes in $\lambda$. The $n = 1$ case is clear. We assume that if $\eta$ is a partition of $k<n$ then $\mathcal{S}\eta$ has no box below an empty position (and hence is a partition by the previous paragraph). 
   
 The inductive proof is broken into three cases. 

The first case is that box $(1, \lambda_1)$ is locked. Then all of the boxes in the first row are also locked. Let $\mu = \lambda \setminus \{\lambda_1 \}$.  A box $(a,b)$ is locked in $\mu$ if and only if the box $(a+1,b)$ is locked in $\lambda$ by Lemma \ref{lockedremove}. Since $|\mu| < |\lambda|$, $\mathcal{S} \mu$ has no box below an empty position. We append the first row back on top of $\mathcal{S} \mu$ to form $\mathcal{S} \lambda$. Hence $\mathcal{S} \lambda$ has no box below an empty position. 

Let $j$ be so that $\lambda_1 = \lambda_j \neq \lambda_{j+1}$.

The second case is when $j >1$. Let $x$ be the box $(j,\lambda_j)$ and let $y = (j-1, \lambda_{j-1})$. Let $\mu = \lambda \setminus \{x \}$. Let $\nu = \lambda \setminus \{x,y \}$. Note that boxes in $\mu$ (and $\nu$) are locked if and only if they are locked in $\lambda$ by Lemma \ref{lockedremove}. Since $|\mu| < |\lambda|$, we can form a partition $\mathcal{S}\mu$ by bringing down all unlocked boxes. Now we place $x$ into the lowest empty position on the ladder of $x$ in $\mathcal{S}\mu$. If the resulting configuration is not a partition, then from above we may assume that there is an empty position above where $x$ was placed. If this is the case, $\mathcal{S}_\lambda(y)$ is below $x$ and has a box $z$ below it (if it didn't then $x$ would move below $y$). But then in $\mathcal{S}\nu$ when we move down all of the boxes we would have $z$ below an empty position (where $y$ is in $\mu$). But $|\nu| < |\lambda|$, so $\mathcal{S}\nu$  does not have an empty position above a box.

Lastly, if $j = 1$ then we let $x = (1, \lambda_1)$. Since $x$ is unlocked, there is at least one empty position in $\mathfrak{L}x$ which has a box above it. Let $\mu = \lambda \setminus \{x \}$. Boxes are locked in $\mu$ if and only if they were locked in $\lambda$ by Lemma \ref{lockedremove}. Since the number of boxes on $\mathfrak{L}x$ was at most the number of boxes on the ladder directly above $x$ in $\lambda$ (this is because $x$ cannot be a type I lock), the number of boxes on $\mathfrak{L}x$ is strictly less than the number on the ladder above $x$ in $\mu$. Therefore, there exists an empty position on $\mathfrak{L}x$ directly below a box in $\mathcal{S} \mu$. We let $(a,b)$ be the lowest such empty position. If $(a,b)$ is the lowest empty position in $\mathcal{S} \mu$ on $\mathfrak{L}x$, then moving $x$ into that position will yield a partition which is obtained from moving $x$ into the lowest empty position in its ladder. If not, then there must be an empty position on $\mathfrak{L}x$ directly below an empty position, all below row $a$. Let $(c,d)$ be the highest such empty position on $\mathfrak{L}x$ below $(a,b)$.  Let $(m_0,d)$ be the lowest box in column $d$ 
($m_0$ is at least $c-\ell$, since there is a box in the position $(c-\ell, d+1)$). 
Let $(m_1, b)$ be the box in column $b$ in $\mathfrak{L}(m_0+1,d)$.
By Corollary \ref{smallestpartition}, all of the boxes of $\mathcal{S}\mu$ are locked.
But $(m_1,b)$ has the position $(m_0+1,d)$ in the same ladder, so if $(m_1,b)$ is locked then there is a box to 
the right of it, in position $(m_1, b+1)$. 
We continue by letting $(m_2, b+1)$ be the box in $\mathfrak{L}(m_0+1,d)$ in column $b+1$. 
Similarly, since this box is locked, there must be a box directly to the right of it which is locked. 
That box will be $(m_2, b+2)$. This process must eventually conclude at step $(m_k,b+k)$ where $b+k$ is at the
end of its row. But then the box $(m_k, b+k)$ must be unlocked, since it has no locked boxes to the right 
and the empty position $(m_0+1, d)$ in the same ladder below it. 
This contradicts all of the boxes of $\mathcal{S} \mu$ being locked.

\end{proof}

\begin{example}
Continuing from the example above $(\lambda = (6,5,4,3,1,1)$ and $\ell = 3)$, we move all of the unlocked boxes down to obtain the smallest partition in $\mathcal{RC} (\lambda)$, which is $\mathcal{S} \lambda = (3,3,2,2,2,2,2,1,1,1,1)$. The boxes labeled $L$ are the ones which were locked in $(6,5,4,3,1,1)$ (and did not move). 

\begin{center}
$\tableau{L&L&L\\
L&L&L\\
L&L\\
L&L\\
L & \mbox{} \\
L & \mbox{}\\
 \mbox{}& \mbox{}\\
 \mbox{}\\
 \mbox{}\\
 \mbox{}\\
 \mbox{}}
$ \end{center}

\end{example}

\begin{theorem}
$\mathcal{S} \lambda$ is the unique smallest partition in its regularization class with respect to dominance order. It can be characterized as being the unique partition (in its regularization class) which has all locked boxes.  
\end{theorem}

\begin{proof}
This follows from Corollary \ref{smallestpartition} and Proposition \ref{algo_yields_partition}. 
\end{proof}

\subsection{The nodes of $ladd_\ell$ are smallest in dominance order}

The nodes of $ladd_\ell$ have been defined recursively by applying the operators $\widehat{f}_{i}$. We now give a simple description which determines when a partition is a node of $ladd_\ell$. 

\begin{proposition} Let $\lambda$ be a partition of $n$. Let $\mathcal{RC}(\lambda)$ be its regularization class. If $\lambda$ is a node of $ladd_\ell$ then $\lambda$ is the smallest partition in $\mathcal{RC}(\lambda)$ with respect to dominance order.
\end{proposition}

\begin{proof} 
The proof is by induction on the size of a partition. Suppose we have a partition $\mu$ in $ladd_\ell$ which is smallest in its regularization class. Equivalently, all of the boxes of $\mu$ are locked. We want to show that $\lambda = \widehat{f_i} \mu$ is still smallest in its regularization class. Let $x = \lambda \setminus \mu$, so that $\lambda$ is just $\mu$ with the addable $i$-box $x$ inserted. There are two cases to consider.

The first is that the insertion of $x$ into $\mu$ makes $x$ into an unlocked box. Since the position $y$ above $x$ must be locked (or $x$ is in the first row), there must exist an empty position $x'$ in $\mathfrak{L}x$ which has a box $y'$ directly above $x'$ in the ladder of $y$. Also, the box $z$ to the left of $x$ must be a type I lock, since $x$ was not in $\mu$ (implying $z$ could not be type II). But then the position $z'$ to the left of $x'$ must have a box, since there is a box in the position above $z'$ (the box to the left of $y'$). Since $\mu$ was smallest in dominance order, it could not have had a $+$ below a $-$ on the same ladder (otherwise the $-$ could be moved down to the $+$, contradicting that $\mu$ was smallest in dominance order). Hence there is no $-$ between the two $+'s$ in positions $x$ and $x'$ of $\mu$. This means that the crystal rule would have chosen to add to $x'$ instead of $x$, a contradiction.  

The second case is that the insertion of $x$ into $\mu$ unlocks a box. If this is the case then let $y$ be the position above $x$. The only case to consider is adding the box $x$ unlocks some box above the row of $x$, on the ladder directly below $x$. Let $z$ denote the lowest such box. If there is a box $x'$ directly above $z$ and no box directly to the right of $x'$ then $x'$ would have been unlocked in $\mu$, since the box $y$ is in $\mu$ but $x$ is not. This contradicts the assumption that $\mu$ was the smallest. 

The only other possibility is that the box $x'$ above $z$  has a box directly to the right of it (this still works even if $z$ is in the first row of the partition).  If there was a box to the right of $z$, then that box would be locked since it was locked in $\mu$. But then $z$ would be locked independently of the addition of $x$. So $z$ is at the end of its row. The position one box to the right of $z$ (let us name it $w$) must therefore be empty and have the same residue as $x$. It is an addable box, so it will contribute a $+$ to the $i$-ladder-signature. There are no $-$ boxes on $\mathfrak{L}x$ below $x$, because if there was a removable $i$-box on $\mathfrak{L}x$ below $x$ then the box $z$ would have been unlocked in $\mu$. Therefore,  no $-$ will cancel the $+$ from position $w$. Since $w$ is on the ladder past $\mathfrak{L}x$, this yields a contradiction as $\widehat{f_i}$ should have added a box to position $w$.

\end{proof}

One can view $reg_\ell$ as having nodes $\{ \mathcal{RC}(\lambda) : \lambda \vdash n, n\geq 0 \}$. 
The $reg_\ell$ model takes the representative $\mathcal{R} \lambda \in \mathcal{RC}(\lambda)$, which happens to be the largest in dominance order. Here, we will take a different representative of $\mathcal{RC}(\lambda)$, the partitions $\mathcal{S} \lambda$, which are smallest in dominance order. One must then give a description of the edges of the crystal graph. We will show in Section \ref{reg_and_crystal} that the crystal operators $\widehat{f}_i$ and $\widehat{e}_i$ are the correct operators to generate the edges of the graph. In other words, we will show that the crystal $ladd_\ell$ constructed above is indeed isomorphic to $reg_\ell$.

\section{Reinterpreting the crystal rule of $reg_\ell$}\label{crystal_lemmas}

In this short section we prove some lemmas necessary for our main theorem (that the crystals $reg_\ell$ and $ladd_\ell$ are isomorphic). We also reinterpret the crystal rule on the classical $reg_\ell$ in terms of the new crystal rule.

\subsection{Two lemmas needed for crystal isomorphism}
\begin{lemma}\label{nodes}
Suppose $\lambda \in ladd_\ell$ and  
 $(a,b)$ and $(c,d)$ are boxes on the same ladder 
such that $(a,b)$ is removable and $(c,d)$ is addable. Then $a < c$.
\end{lemma}

\begin{proof}
If there was a $-$ box above an empty $+$ box on the same ladder, then one could move that $-$ box 
down to the $+$ position to form a new partition in the same regularization class. But this cannot happen
 since the nodes of $ladd_\ell$ are smallest in dominance order.
\end{proof}

\begin{lemma}\label{ell_regular_nodes}
Suppose $\lambda \in ladd_\ell$ and  
 $(a,b)$ and $(c,d)$ are boxes on the same ladder 
such that $(a,b)$ is addable and $(c,d)$ is removable. Then $a < c$.
\end{lemma}
\begin{proof}
Since $\lambda$ is in $reg_\ell$, it is $\ell$-regular. Suppose such positions $(a,b)$ and $(c,d)$ exist. Then the rim hook starting at box $(a, b-1)$ and following the border of $\lambda$ down to the $(c,d)$ box will cover exactly $(b-d) \ell$ boxes. However, this border runs over only $b-d$ columns, so there must exist a column which contains at least $\ell$ boxes at the ends of their rows. This contradicts $\lambda$ being $\ell$-regular. 
\end{proof}

\subsection{Reinterpreting the classical crystal rule}

The next theorem proves that the classical crystal rule can be interpreted in terms of ladders. In fact 
the only difference between the classical rule and the ladder crystal rule is that ladders 
are read bottom to top instead of top to bottom.

\begin{theorem}\label{regular_rule}
The $i$-signature (and hence reduced $i$-signature) of an $\ell$-regular partition $\lambda$ in 
$reg_\ell$ can be determined by reading from its leftmost ladder to rightmost ladder, reading each ladder from bottom to top.
\end{theorem}

\begin{proof}
If positions $x$ and $y$ contain $+'s$ (or $-'s$) of the same residue with the ladder 
of $x$ to the left of the ladder of $y$, then by the regularity of $\lambda$, $x$ will 
be in a row below $y$. Hence reading the $i$-signature up ladders from left to right is 
equivalent to reading up the rows of $\lambda$ from bottom to top.

\end{proof}

\begin{example}
$\lambda = (6,5,3,3,2,2,1)$ and $\ell = 3$. Suppose we wanted to find the 2-signature for $\lambda$. 
Then we could read from leftmost ladder (in this case, the leftmost ladder relevant to the 2-signature
 is the one which contains position (8,1)) to rightmost ladder. Inside each ladder we read from bottom 
to top . The picture below shows the positions which 
correspond to addable and removable 2-boxes, with their ladders. The 2-signature is $+---$.

\begin{center}
$\begin{array}{cc}
\tableau{\mbox{}&\mbox{}&\mbox{}&\mbox{}&\mbox{}&-\\
\mbox{}&\mbox{}&\mbox{}&\mbox{}&\mbox{}\\
\mbox{}&\mbox{}&\mbox{}\\
\mbox{}&\mbox{}&-\\
\mbox{}&\mbox{}\\
\mbox{}&-\\
\mbox{}}
&

\put(-112,-122){$+$}

\put (-112,-122){\line(1,2){70}}
\put (-85,-122){\line(1,2){70}}

\end{array}
$
\end{center}
\end{example}

\section{Regularization and Crystal Isomorphism}\label{reg_and_crystal}
The results of this section come from ideas originally sketched out with Steve Pon in the summer of 2007.
\subsection{Crystal isomorphism}\label{crystal_isomorphism}
We will now prove that our crystal $ladd_\ell$ from Section \ref{new_crystal} is indeed isomorphic 
to the crystal $reg_\ell$.

\begin{remark}
In the following definitions, any positions $(i,j)$ for either of $i,j = 0$ are assumed to be \textit{in} the diagram for the partition $\lambda$. 
\end{remark}

\begin{definition}
The $k^{th}$ ladder of $\lambda$ will refer to all of the positions $(i,j)$ with $i,j \geq 0$ of $\lambda$ which are on $\mathfrak{L}(k,1)$. $|\mathfrak{L}(k,1)|$ will denote the number of positions of this ladder which are in $\lambda$ (so $|\mathfrak{L}(k,1)|$ depends on the partition $\lambda$).
\end{definition}

\begin{definition}
We let $\mathcal{A}_k (\lambda)$ denote the boxes $(i,j)$ ($i,j \geq 0$) on the $k^{th}$ ladder which have boxes $(i,j+1)$ and $(i+1,j)$ in $\lambda$, but $(i+1,j+1)$ not in $\lambda$. Similarly, we let $\mathcal{B}_k(\lambda)$ denote the boxes $(i,j)$ on the $k^{th}$ ladder which do not have boxes $(i,j+1)$ and $(i+1,j)$ in $\lambda$.
\end{definition}

\begin{remark}
Note that the boxes in $\mathcal{B}_k(\lambda)$ are exactly those boxes on the $k^{th}$ ladder which are removable in $\lambda$ and that the boxes $(i,j)$ in $\mathcal{A}_k(\lambda)$ are exactly those boxes for which there is an addable position in $(i+1,j+1)$ on the $(k+\ell)^{th}$ ladder.
\end{remark}

\begin{lemma}\label{b-a}
Let $\lambda$ and $\mu$ be two partitions in the same $\ell$ regularization class. Then $|\mathcal{B}_k(\lambda)| - |\mathcal{A}_k(\lambda)| = |\mathcal{B}_k(\mu)| - |\mathcal{A}_k(\mu)|$. In other words, $|\mathcal{B}_k| - |\mathcal{A}_k|$ is an invariant of a regularization class.
\end{lemma}

\begin{proof}
The number of boxes in any ladder of a partition is clearly an invariant of a regularization class (in fact this can be the definition of a regularization class). We will count all of the boxes in the $k^{th}$ ladder. Each box in the $k^{th}$ ladder falls into at least one of three different categories. Either it has a box directly to the right of it, a box directly below it, or neither. Boxes with neither are counted by $|\mathcal{B}_k(\lambda)|$. Boxes with a box below are counted in $|\mathcal{L}(k+1,1)|$. Boxes with a box to the right are counted in $|\mathcal{L}(k+\ell-1,1)|$. But we have over-counted as some boxes can have a box both directly below and to the right. To fix this, we must subtract by  $|\mathcal{L}(k+\ell,1)|$ (those boxes $(i,j)$ which also have a box $(i+1,j+1)$ in $\lambda$) and $|\mathcal{A}_k(\lambda)|$ (those boxes which do not have $(i+1,j+1)$ in $\lambda$). Hence:
$$ |\mathcal{L}(k,1)| =  |\mathcal{L}(k+1,1)| +  |\mathcal{L}(k+\ell-1,1)| +  |\mathcal{B}_k(\lambda)| -  |\mathcal{L}(k+\ell,1)| -  |\mathcal{A}_k(\lambda)|.$$

The lemma follows, as all terms other than $|\mathcal{B}_k(\lambda)| - |\mathcal{A}_k(\lambda)|$ are invariants of the regularization class.
\end{proof}

\begin{example}
Let $\lambda = (3,2,1)$ and $\ell = 3$. There are $4$ partitions in the regularization class of $\lambda$: $(3,2,1), (3,1,1,1), (2,2,2)$ and $(2,1,1,1,1)$. Let $k = 4$. There are $4$ total positions in the ladder containing $(4,1)$ (they are $(0,3), (2,2), (4,1)$ and $(6,0)$). Then $\mathcal{L}(4,1) = 3$, $\mathcal{L}(5,1) = 2$, $\mathcal{L}(6,1) = 2$ and $\mathcal{L}(7,1) = 1$.  $\mathcal{A}_4((3,2,1)) = \{ (0,3)\}$, $\mathcal{B}_4((3,2,1)) = \{(2,2) \}$, $\mathcal{A}_4((3,1,1,1)) = \{(0,3) \}$, $\mathcal{B}_4((3,1,1,1)) = \{(4,1) \}$ and $\mathcal{A}_4$, $\mathcal{B}_4$ are empty for the other two partitions. $|\mathcal{B}_4| - |\mathcal{A}_4| = 0$ for all $4$ partitions. Also, $ \mathcal{L}(4,1)=3= 2+2-1+0 =\mathcal{L}(5,1) + \mathcal{L}(6,1) - \mathcal{L}(7,1) + ( |\mathcal{B}_4| - |\mathcal{A}_4|)$, as was shown in the proof of Lemma \ref{b-a}.

$$\begin{array}{cccc} \tableau{\mbox{}&\mbox{}&\mbox{}\\\mbox{}&\mbox{}\\\mbox{}} & \tableau{\mbox{}&\mbox{}&\mbox{}\\\mbox{}\\\mbox{}\\\mbox{}}&\tableau{\mbox{}&\mbox{}\\\mbox{}&\mbox{}\\\mbox{}&\mbox{}} & \tableau{\mbox{}&\mbox{}\\\mbox{}\\\mbox{}\\\mbox{}\\\mbox{}}\end{array}$$
\end{example}

\begin{lemma}\label{ladderlemma}  Let $\lambda$ be a node of $ladd_\ell$. Then the number of ladder-(co)normal boxes on the $k^{th}$ ladder of $\lambda$ is the same as the number of (co)normal boxes on the $k^{th}$ ladder of $\mathcal{R}\lambda$. In particular, a ladder-good (ladder-cogood) $i$-box of $\lambda$ lies on
 the same ladder as the good (cogood) $i$-box of $\mathcal{R} \lambda$.
\end{lemma}

\begin{proof}

 Lemma \ref{nodes} implies that there is no cancelation in the ladder $i$-signature on a ladder of $\lambda$.
Similarly, Lemma \ref{ell_regular_nodes} implies that there is no cancelation in the $i$-signature on a ladder of $\mathcal{R}\lambda$. This allows us to calculate the ladder $i$-signature
 (resp. $i$-signature) of $\lambda$ (resp. $\mathcal{R}\lambda$) by counting the number of addable and removable 
$i$-boxes in each ladder.

First, we start with two adjacent ladders, $\mathcal{L}(k,1)$ and $\mathcal{L}(k+\ell,1)$. The difference between the number of $-$'s which contribute from $\mathcal{L}(k,1)$ and the number of $+$'s which contribute from $\mathcal{L}(k+\ell,1)$ is an invariant, $\alpha_k$, by Lemma \ref{b-a}. If $\alpha_k$ is positive, then after cancellation between these two ladders, there are exactly $\alpha_k$ $-$'s remaining in ladder $\mathcal{L}(k,1)$. If $\alpha_k$ is negative, then there are $-\alpha_k$ $+$'s remaining in ladder $\mathcal{L}(k+\ell,1)$. This is independent of calculating the ladder $i$-signature of $\lambda$, or the $i$-signature of $\mathcal{R}\lambda$. Continuing this process between all ladders, we see that the number of uncanceled $+$ and $-$ signs on each ladder is an invariant, and the lemma follows.

\end{proof}

\begin{corollary}
Let $\lambda$ be a node in $ladd_\ell$. Then $\widehat{\varphi}_i(\lambda) = \varphi_i(\mathcal{R}\lambda)$ and $\widehat{\varepsilon}_i(\lambda) = \varepsilon_i(\mathcal{R}\lambda)$.
\end{corollary}

\begin{theorem}\label{crystal_commutes}
Regularization commutes with the crystal operators. In other words:

\begin{enumerate}
\item $\mathcal{R} \circ \widehat{f}_i = \tilde{f}_i \circ \mathcal{R}$,
\item $\mathcal{R} \circ \widehat{e}_i = \tilde{e}_i \circ \mathcal{R}$.
\end{enumerate}

\end{theorem}

\begin{proof}
(1) will follow from Lemma \ref{ladderlemma}, since applying an $\widehat{f}_i$ to a partition $\lambda$ will place an $i$-box in the same ladder of $\lambda$ as $\tilde{f}_i$ places in $\mathcal{R} \lambda$. (2) follows similarly.
\end{proof}

\begin{corollary}\label{crystalisisom}
The crystal $reg_\ell$ is isomorphic to $ladd_\ell$.
\end{corollary}

\begin{proof}
The map $\mathcal{R} : ladd_\ell \to reg_\ell$ gives the isomorphism.
The map $\mathcal{S}$ described in Section \ref{locked} is the inverse of $\mathcal{R}$.
The other crystal isomorphism axioms are routine to check.
\end{proof}

\begin{example}
Let $\lambda = (2,1,1,1)$ and $\ell = 3$. Then $\mathcal{R} \lambda = (2,2,1)$, $\widehat{f}_2 \lambda = (2,1,1,1,1)$ and $\tilde{f}_2 (2,2,1) = (3,2,1)$. But $\mathcal{R} (2,1,1,1,1) = (3,2,1)$.

\[
\begin{CD}
(2,1,1,1) @>\widehat{f}_2>>(2,1,1,1,1)\\
@VV\mathcal{R}V @VV\mathcal{R}V\\
(2,2,1) @>\widetilde{f}_2>> (3,2,1)
\end{CD}
\]

\end{example}

\section{The nodes of $ladd_\ell$}\label{MullMap}

\subsection{Characterization via hook lengths}
 It was pointed out to the author by Fayers that a characterization of the nodes of $ladd_\ell$ can be described in terms of hook lengths and arm lengths. We now include this characterization.
 
 \begin{theorem} A partition $\lambda$ belongs to the crystal $ladd_\ell$ if and only if there does not exist a box $(i,j) \in \lambda$ such that $h_{(i,j)}^\lambda = \ell \cdot \mathrm{arm}(i,j)$.
 \end{theorem}
 
 \begin{proof}
 If $\lambda$ has a box $(i,j)$ with hook length $\ell \cdot \mathrm{arm}(i,j)$ then the box $(i, \lambda_i)$ will 
be unlocked, as it is on the same ladder as the empty position directly below the last box in column $j$. 
Hence $\lambda$ is not in $ladd_\ell$. 
 
 If $\lambda$ does not belong to the crystal $ladd_\ell$ then there exists an unlocked box $(a,b)$ in
 the diagram of $\lambda$. This implies that there exists 
an unlocked box $(i,\lambda_i)$ which is either directly below a locked box or in row $1$. 
This box is then unlocked because it violates a type I lock rule, which means that in the same ladder as
 $(i,\lambda_i)$ there is an empty position $(n,j)$ directly below a box $(n-1,j)$. But then the box $(i,j)$
 will satisfy $h_{(i,j)}^\lambda = \ell \cdot \mathrm{arm}(i,j)$.
 \end{proof} 

\begin{section}{Acknowledgements}
The author would like to thank several people who made this paper possible. I first showed the idea of the ladder crystal to Steve Pon in the summer of 2007 and with him we sketched out some work that would later be the basis for proving that the two crystals were isomorphic. Later on, I ran into a problem showing the existence of a smallest partition in dominance order in each regularization class and together with Brant Jones we developed a construction for finding such a partition, which I've generalized in this paper in Section \ref{locked}. After sending a draft to Matthew Fayers, he responded with helpful suggestions and corrections. His comments also led me to write Section \ref{MullMap}. Finally, I would like to thank my advisor Monica Vazirani. Finding answers to her questions led me to discover the ladder crystal, and numerous ensuing discussions helped in the coherency and direction of this paper.
\end{section}

\bibliographystyle{amsalpha}

\end{document}